\documentclass[a4paper,10pt]{article}
\usepackage[utf8]{inputenc}
\usepackage{amsfonts, amssymb, amsmath, amsthm}
\usepackage{mathtools}
\usepackage{graphicx}
\usepackage[inline]{enumitem}

\usepackage{a4wide}
\usepackage{fullpage}
\usepackage{authblk}

\usepackage{hyperref}
\hypersetup{colorlinks=true, linkcolor=red, urlcolor=blue, citecolor=blue, pdfpagemode=UseNone, pdfstartview=FitH, bookmarksopen=true}
\hypersetup{pdftitle=Title}

\author{Pavel Paták}
\affil{Department of Mathematics and Statistics, Masaryk University, Brno, Czech Republic}

\newtheorem*{theorem*}{Theorem}
\newtheorem{theorem}{Theorem}
\newtheorem*{lemma*}{Lemma}
\newtheorem{lemma}[theorem]{Lemma}
\newtheorem*{proposition*}{Proposition}
\newtheorem{proposition}[theorem]{Proposition}
\newtheorem*{fact*}{Fact}

\newtheorem*{question*}{Question}

\newtheorem*{corollary*}{Corollary}

\newcounter{claimcounter}[theorem]
\numberwithin{claimcounter}{theorem}
\newtheorem*{claim*}{Claim}
\newtheorem{claim}[claimcounter]{Claim}

\theoremstyle{remark}
\newtheorem*{remark*}{Remark}

\theoremstyle{definition}
\newtheorem*{definition*}{Definition}

\numberwithin{equation}{section}

\newtheorem*{observation*}{Observation}

\newtheorem{conjecture}{Conjecture}

\newcommand{\R}{\mathbb{R}}

\DeclareMathOperator{\conv}{conv}
\DeclareMathOperator{\aff}{aff}
\DeclareMathOperator{\cl}{cl}
\DeclareMathOperator{\rk}{rk}

\title{Tverberg type theorems for matroids}
\author{Pavel Paták}

\begin{document}

\maketitle

\begin{abstract}
In this paper we show a variant of colorful Tverberg's theorem which is valid in any matroid:
Let $S$ be a sequence of non-loops in a matroid $M$ of finite rank $m$ with closure operator $\cl$. Suppose that $S$ is colored
in such a way that the first color does not appear more than $r$-times and each other color appears at most $(r-1)$-times.
Then $S$ can be partitioned into $r$ \emph{rainbow} subsequences $S_1,\ldots, S_r$ such that
$\cl \emptyset\subsetneq \cl S_1\subseteq \cl S_2\subseteq \ldots \subseteq \cl S_r$. In particular, $\emptyset\neq \bigcap_{i=1}^r \cl S_i$.
A subsequence is called rainbow if it contains each color at most once.

The conclusion of our theorem is weaker than the conclusion of the original Tverberg's theorem in $\R^d$, 
which states that $\bigcap \conv S_i\neq \emptyset$, whereas we only claim that $\bigcap \aff S_i\neq \emptyset$.
On the other hand, our theorem strengthens the Tverberg's theorem in several other ways:
\begin{enumerate*}[label={\roman*})]
\item it is applicable to any matroid (whereas Tverberg's theorem can only be used in $\R^d$), 
\item instead of $\bigcap \cl S_i\neq \emptyset$ we have 
the stronger condition $\cl \emptyset\subsetneq \cl S_1\subseteq \cl S_2\subseteq \ldots \subseteq \cl S_r$, and
\item we add color constraints that are even stronger than the color constraints in the colorful version of Tverberg's theorem.
\end{enumerate*}
  
Recently, we used the first property and applied the non-colorful version of this theorem to homology groups 
with $GF(p)$ coefficients to obtain several non-embeddability results, for details we refer to~\cite{ourKuhnel}.
\end{abstract}

\section{Introduction}

Tverberg's theorem~\cite{Tverberg} states that given $(d+1)(r-1)+1$ points\footnote{We allow repetitions among these points.} 
in $\R^d$, it is possible to split these points into $r$ sets
$S_1,\ldots, S_r$ with intersecting convex hulls, that is with 
$\bigcap \conv S_i\neq \emptyset$.

If one replaces convex hulls with affine hulls, one obtains a valid statement (Lemma~\ref{lem:affine_tverberg}),
which has the advantage of being applicable to any field~\cite{ourKuhnel,ourKuhnelArXiv}.
Lemma~\ref{lem:affine_tverberg} is also easier to prove than the original Tverberg's theorem. 
Since the proof only uses properties of closure operators, the statement does generalize
to any matroid (Lemma~\ref{lem:matroidal_tverberg}).
In both these cases the conclusion can be strengthened a bit: instead of $\cl S_1\cap\ldots\cl S_r\neq \emptyset$,
one can require $\cl \emptyset \subsetneq \cl S_1\subseteq \cl S_2\subseteq \ldots \subseteq  \cl S_r$.

In this paper we study the variant of Tverberg's theorem for matroidal closures
and show that it allows a colorful version -- a generalization where the original points
are colored and one furthermore requires that no resulting set $S_i$, $i=1,\ldots, r$
contains two or more points of the same color.

While the version without colors is straightforward~\cite[Lemma 12]{ourKuhnelArXiv} 
the proof of the colorful version
is more subtle. 
Moreover, our proof method yields an efficient algorithm that finds the required sets in polynomial time.

\subsection{Terminology}\label{subsec:terminology}
Before we state our results formally, let us introduce some notations and terminology
which will allow us to nicely present the statements and proofs.
We assume that the reader is acquainted with the basic matroid theory.
We always use the symbols $r$ and $m$ to denote non-negative integers.
We use the symbols $\cl$, $\aff$, $\conv$ and $\rk$ for matroidal closure, affine closure, convex hull
and rank function, respectively.

If $M$ is a set, we consider a sequence $S=(m_i)_{i\in I}$ of elements from $M$
as a set of pairs $\{(i,m_i)\mid i\in I\}$.
With this convention we can use the set theoretic terminology for sequences: 
$|S|$ is the length of the sequence, $S'\subseteq S$ means that $S'$ is a subsequence of $S$,
we know what it means for two subsequences to be disjoint,
we can use the operation $S\setminus S'$ of (sequence) difference, etc.

If $S=\{(i,m_i)\mid i\in I\}$ is a sequence and we need to refer to the set $\{m_i\mid i\in I\}$,
we use the symbol $S^{set}$.

If $\Psi$ is a map from the subsets of $M$ (for example a closure operator, rank function),
and $S=(m_i)_{i\in I}$ is a sequence in $M$, we use a shorthand $\Psi(S):=\Psi\left(S^{set}\right)$.
To make formulas and equations shorter, we leave out the parantheses after the operators $\cl$, $\aff$, $\conv$ and $\rk$
when there is no danger of confusion.

A coloring of a sequence $S=\{(i,m_i)\mid i\in I\}$ is any map $c\colon S\to C$ into some set $C$ of colors,
that is, $c$ assigns to each pair $(i,m_i)$ a color from $C$. The sequence $S$ is \emph{rainbow} with respect to $c$,
if the restriction of $c$ to $S$ is injective.

\subsection{Main results}

\bigskip
Let us first state the non-colorful variant of Tverberg's theorem for affine hulls
and its easy generalization to matroidal closures.

\begin{lemma}[Tverberg for affine hulls~\cite{ourKuhnel,ourKuhnelArXiv}]\label{lem:affine_tverberg}
 Let $S$ be a sequence of points in an affine space $\mathbb A$
 of dimension $d$. If\footnote{We do not require $d$ to be finite,
 therefore the slightly unusual formulation.} $|S|>(d+1)(r-1)$, 
 then there exist $r$ pairwise disjoint subsequences $S_1,\ldots, S_r$ of $S$ with $\bigcap_{i=1}^r \aff S_i\neq \emptyset$.
 In fact, there are $r$ pairwise disjoint subsequences $S_1,\ldots, S_r$ satisfying 
 $\emptyset\neq \aff S_1\subseteq \aff S_2\subseteq \ldots \subseteq \aff S_r$.
\end{lemma}

\begin{lemma}[Matroidal Tverberg]\label{lem:matroidal_tverberg}
 Let $M$ be a (finitary\footnote{Finitary matroids are generalization
 of matroids to not necessary finite ground sets. They add the following axiom to the usual axioms for finite matroids:
 If $y\in\cl(X)$, then there exists a finite set $X'\subseteq X$ such that $y\in \cl(X')$. With these addition,
 such terms as rank or basis can be correctly defined.}) matroid of rank $m$ with closure operator $\cl$ and $S$ be a sequence of points in $M$ 
 with $|S|>m(r-1)$.
 Then there exist $r$ pairwise disjoint subsequences $S_1,\ldots, S_r$ of $S$
 satisfying $\cl\emptyset\subsetneq \cl S_1\subseteq \cl S_2\subseteq \ldots \subseteq \cl S_r$.
\end{lemma}

In~\cite[Lemma~13]{ourKuhnelArXiv} we only stated that there exists sets $S_i$ with $\emptyset\neq \bigcap \aff S_i$.
However, the proof there implies Lemma~\ref{lem:affine_tverberg}, and (if one replaces $\aff$ with the closure
operator $\cl$ of a matroid) Lemma~\ref{lem:matroidal_tverberg}.
In the case of matroids of finite rank, both lemmas can also be obtain as a direct consequence of Theorem~\ref{thm:colorful_matroidal_tverberg}.

In~\cite{ourKuhnel,ourKuhnelArXiv} we applied Lemma~\ref{lem:affine_tverberg} to homology groups over finite fields. This enabled us to prove some inequalities for simplicial complexes embeddable into various manifolds.
Our colorful matrodial Tverberg (Theorem~\ref{thm:colorful_matroidal_tverberg}) provides a control
of the resulting sets, which enables us to further improve the bounds from~\cite{ourKuhnel,ourKuhnelArXiv}. For the details
of the improvement, see the author's thesis~\cite{patak:thesis}.

We are now ready to state the main results of this paper.

\begin{theorem}\label{thm:colorful_matroidal_tverberg}
 Let $M$ be a matroid of a finite rank $m$ and $S$ be a sequence of non-loops in $M$ 
 colored by some colors in such a way that at most $m$ elements of $S$ are colored by the first color,
 at most $m-1$ by the second color, at most $m-1$ by the third color, etc.
 If $|S|>m(r-1)$, 
 then there exist $r$ pairwise disjoint rainbow subsequences $S_1,\ldots, S_r$ of $S$,
 such that $\cl\emptyset \subsetneq \cl S_1\subseteq \cl S_2\subseteq \ldots \subseteq \cl S_r$.
 
 Furthermore, if the time required to decide whether a point $x\in M$ lies in the closure of a set $Y\subseteq M$ is bounded by $u$,
 then the subsequences $S_1,\ldots, S_r$ can be found in time polynomial in $r$, $m$, $u$ and $|S|$.
\end{theorem}

In the proof of Theorem~\ref{thm:colorful_matroidal_tverberg} we encounter another version
of colorful matroidal Tverberg's theorem.

\begin{theorem}\label{thm:colored-tverberg-special}
 Let $M$ be a matroid of a finite rank $m$ and $S$ a sequence of non-loops in $M$
 colored by $m$ colors in such a way that at least $r$ elements of $S$ are colored by the first color,
 at least $r-1$ by the second color, at least $r-1$ by the third, \dots, at least $r-1$ by the $m$th color. 
 
 Then there exist $r$ pairwise disjoint rainbow subsequences $S_1,\ldots, S_r$ of $S$
 such that $\cl\emptyset \subsetneq \cl S_1\subseteq \cl S_2\subseteq \ldots \subseteq \cl S_r$.
 
 Furthermore, if the time required to decide whether a point $x\in M$ lies in the closure of a set $Y\subseteq M$ is bounded by $u$,
 then the subsequences $S_1,\ldots, S_r$ can by found in time polynomial in $r$, $m$, $u$ and $|S|$.
\end{theorem}
Note the different conditions on the number of points of each color.
In Theorem~\ref{thm:colorful_matroidal_tverberg} these conditions are used to ensure
that we have enough colors. 
In Theorem~\ref{thm:colored-tverberg-special} we already have the right number of colors,
but the conditions ensure that the length of $S$ is sufficient.

\medskip
Moreover, these results are tight:
\begin{proposition}\label{prop:tight}
 Lemma~\ref{lem:affine_tverberg}, Lemma~\ref{lem:matroidal_tverberg}, Theorem~\ref{thm:colorful_matroidal_tverberg}
 and Theorem~\ref{thm:colored-tverberg-special} are sharp. To be precise, for any $r$ and any matroid $M$ of rank $m$ there
 exists a sequence $S$ of non-loops in $M$ with $|S|=m(r-1)$ such that any division of $S$ into  $r$ disjoint subsequences $S_1$,\dots, $S_r$ satisfies $\bigcap \cl S_i =\cl\emptyset$.
\end{proposition}

\textbf{Tverberg-type theorems in $\R^d$}\\
Let us now compare our main results with the related theorems valid in $\R^d$.

In this section $\Delta_n$ denotes the $n$-dimensional simplex.

Tverberg's theorem can be stated as follows: If $f\colon \Delta_{(d+1)(r-1)}\to \R^d$ is an affine map,
there are $r$ pairwise disjoint faces $\sigma_1,\ldots, \sigma_r$ of $\Delta_{(d+1)(r-1)}$
with $\bigcap_{i=1}^r f(\sigma_i)\neq \emptyset$. This is the reason why 
Tverberg's theorem is also called affine Tverberg's theorem. To avoid confusion, we have decided not to use the name
``affine Tverberg'' for Lemma~\ref{lem:affine_tverberg}.

If $r$ is a prime power, Özaydin~\cite{Ozaydin:Tverberg} showed that the same result holds for an arbitrary continuous map $f$.
The statement is known as topological Tverberg. It was a long-standing open problem, whether topological Tverberg can be extended to other values of $r$.
The negative answer came in 2015, when Frick (based on the previous work of Mabillard and Wagner~\cite{Mabillard})
constructed first counterexamples~\cite{Frick}. Counterexamples for other values of $d$ and $r$ followed shortly afterwards.~\cite{Uli:TverbergThree}

If $r$ is a prime, there is a colorful version of (topological) Tverberg's theorem~\cite{Blagojevic:Optimal_Tverberg} as well:
Suppose that the vertices of $K=\Delta_{(d+1)(r-1)}$ are colored in such a way, that no color is used
more than $(r-1)$-times. 
Then for every continuous map $f\colon \Delta_{(d+1)(r-1)}\to \R^d$, there are $r$ pairwise disjoint rainbow\footnote{Containing
each color at most once.}
faces $\sigma_1,\ldots, \sigma_r$ of $\Delta_{(d+1)(r-1)}$ with $\bigcap_{i=1}^r f(\sigma_i)\neq \emptyset$.

The colorful version provides more control over the resulting sets $\sigma_1,\ldots,\sigma_r$.

Even if $f$ is an affine map, the only known proof uses topological methods and needs the assumption that $r$ is prime.
Whether this assumption can be relaxed in the affine situation is an open question.
Moreover, the topological proof does not provide any way how to find the pairwise disjoint faces, it merely shows their existence.

We see that Theorem~\ref{thm:colorful_matroidal_tverberg} does not require $r$ to be a prime number, it relaxes the conditions on the colors from topological version a bit and provides an efficient algorithm for finding the desired sets.

We also note that Bárány, Kalai and Meshulam proved another, very different 
Tverberg Type Theorem for Matroids~\cite{anotherMatroidalTverberg},
they considered continuous maps from the matroidal complex and showed the following:
If $b(M)$ denotes the maximal number of disjoint bases in a matroid M of rank $d+1$,
then for any continuous map $f$ from the matroidal complex $M$ into $\R^d$ there
exists $t\geq \sqrt{b(M)}/4$ disjoint independent sets $\sigma_1,\ldots, \sigma_t\in M$ such that
$\bigcap_{i=1}^{t}f(\sigma_i)\neq\emptyset$.

\section{Tightness}
We postpone the technical proofs of our main results, Theorems~\ref{thm:colorful_matroidal_tverberg} and \ref{thm:colored-tverberg-special} to the end of the paper. First we prove Proposition~\ref{prop:tight} showing their tightness. 
The proof is a variant of the standard construction for showing that Tverberg's theorem is tight.

We start with an auxiliary lemma.
 \begin{lemma}\label{lem:intersection}
 Let $M$ be a matroid with finite basis $B$. Then for any two sets $U,V\subseteq B$
            \begin{equation}
	      \cl(U)\cap \cl(V) =\cl(U\cap V).\label{eq:intersection}
            \end{equation}
 \end{lemma}
 \begin{proof}
  Since the operator $\cl$ is monotone, the inclusion
  $\cl(U\cap V)\subseteq \cl(U)\cap \cl(V)$ is obvious.
  Let us now prove the opposite inclusion.
  
  Let $x\in \cl(U)\cap \cl(V)$ be an arbitrary element. We want to show that $x\in \cl(U\cap V)$.
  If $x$ is a loop, $x\in \cl \emptyset \subseteq \cl(U\cap V)$. So assume that $x$ is not a loop.
  
  Let $U'\subseteq U$ and $V'\subseteq V$ be inclusion minimal subsets with $x\in \cl(U')$ and $x\in \cl(V')$, respectively.
  Since we assume that $x$ is not a loop, $U'\neq \emptyset \neq V'$.
  
  We will show by contradiction that $U'=V'$, hence proving the claim.
  If $U'\neq V'$, we may up to symmetry assume that there is
  an element $u'\in U'$ which does not lie in $V'$.

  From the inclusion minimality of $U'$ follows that $x\in \cl\Bigr((U'\setminus \{u'\})\cup \{u'\}\Bigr)\setminus \cl(U'\setminus \{u'\})$. 
  The exchange principle yields $u'\in \cl\Bigl(U'\setminus \{u'\}\cup \{x\}\Bigr)$.
  Similarly $v'\in \cl\Bigl(V'\setminus \{v'\}\cup \{x\}\Bigr)$ for an arbitrary $v'\in V'$.
  
  The set $U'\cup V'$ is independent being a subset of a basis $B$.
  By construction $\cl\Bigl(U'\setminus\{u'\}\cup V'\setminus\{v'\}\cup \{x\}\Bigr)=\cl(U'\cup V')$.
  Comparing the ranks of both sides and using the fact that $u'\notin V'$, we see that $v'$ has to belong to $U'$. 
  Since $v'$ was arbitrary, this implies $V'\subsetneq U'$ -- in contradiction with $U'$ being minimal with $x\in \cl(U')$.
 \end{proof}

We can now finally prove Proposition~\ref{prop:tight}.
\begin{proof}[Proof of Proposition~\ref{prop:tight}]
 Let $B=(e_1,\ldots,e_m)$ be a basis of the matroid $M$.
 It suffices to take 
 \begin{equation}
  S=\underbrace{e_1,e_1,\ldots, e_1}_{(r-1)\times}, \underbrace{e_2,e_2,\ldots, e_2}_{(r-1)\times},\ldots, \underbrace{e_m,e_m,\ldots, e_m}_{(r-1)\times}.
 \end{equation}
 Let $S_1,\ldots, S_r$ be disjoint subsequences of $S$.
 Then
 \[
  \bigcap_{i=1}^r \cl(S_i) =  \cl\left(\bigcap_{i=1}^r S_i^{set}\right)=\cl(\emptyset),
 \]
 where the first equality follows by inductive application of Lemma~\ref{lem:intersection} and the second equality uses the fact that each element $e_j$ is missing in at least one sequence $S_i$.
\end{proof}

We also note that the assumption in Theorem~\ref{thm:colorful_matroidal_tverberg} that there are at most $r$ points of the first color is necessary.
Otherwise, one can consider the sequence $S=(1,2,3,4,\ldots, n, n+1)$ in $\R^1$ where the first $n$ elements are red
and the last element $n+1$ is blue. Then although the length of $S$ can be arbitrary,
there are no three disjoint rainbow subsequences $S_1$, $S_2$, $S_3$ with 
$\aff S_1\cap \aff S_2\cap \aff S_3\neq \emptyset$.
On the other hand, it is not true that this condition is necessary in every matroid. 
For example, consider the affine line over the field with two elements.

\bigskip
\section{The proof}
We begin the proof by showing that Theorem~\ref{thm:colored-tverberg-special} implies
Theorem~\ref{thm:colorful_matroidal_tverberg}. 

The reduction of Theorem~\ref{thm:colorful_matroidal_tverberg} to Theorem~\ref{thm:colored-tverberg-special} 
follows a well known pattern, a similar reduction previously appeared in the proof 
of the optimal colored Tverberg theorem~\cite{Blagojevic:Optimal_Tverberg} or 
in Sarkaria's proof for the prime power Tverberg theorem~\cite[2.7.3]{Sarkaria:Tverberg}, 
see also de Longueville's exposition~\cite[Prop.~2.5]{deLongueville:exposition}.
Nevertheless, there are subtle differences because we are working in greater generality
and because we need to take algorithmic aspects into consideration.

\begin{proof}[Theorem~\ref{thm:colored-tverberg-special} implies Theorem~\ref{thm:colorful_matroidal_tverberg}.]
Assume that the assumptions of Theorem~\ref{thm:colorful_matroidal_tverberg} are satisfied.
We show how to turn the sequence $S$ and the matroid $M$ with closure operator $\cl$ into a sequence $S'$ and matroid $M'$
with closure operator $\cl'$ that satisfy the assumptions of Theorem~\ref{thm:colored-tverberg-special}.
Moreover, we construct $S',M',\cl'$ and the coloring of $S'$ in a such way that 
the sets $S_1:=S_1'\cap S$, $S_2:=S_2'\cap S$, \dots, $S_r:=S_r'\cap S$
will satisfy $\cl\emptyset\subsetneq\cl(S_1)\subseteq \cl(S_2)\subseteq \ldots \subseteq \cl(S_r)$
iff and only if $\cl'\emptyset\subsetneq\cl'(S'_1)\subseteq \cl'(S'_2)\subseteq \ldots \subseteq \cl'(S'_r)$
and the rainbowness of $S_i'$ will imply that $S_i$ is rainbow. 

Let $m$ be the rank of $M$ and $d$ the number of colors used in $S$.
From the conditions follows that $d-m\geq 0$.

If the length of $S$ is strictly larger than $m(r-1)+1$,  we throw the superfluous elements of $S$ away. 
This does not add a point of any color, therefore all assumptions of Theorem~\ref{thm:colorful_matroidal_tverberg} remain preserved. 
So we may assume that the length of $S$ is precisely $m(r-1)+1$.

We form $M'$ from $M$ by adding $(d-m)$ new coloops\footnote{A coloop is an element $x$
that is independent on any set that does not contain $x$. In other words,
we form $M'$ as the direct sum of $M$ with the uniform matroid $U_{d-m}^{d-m}$.} $x_1,\ldots, x_{d-m}$. 
Now we form the sequence $S'$ by appending $(\underbrace{x_1,x_1, \ldots,x_1}_{(r-1)\times}, \underbrace{x_2,x_2, \ldots,x_2}_{(r-1)\times},\ldots \underbrace{x_{d-m}, \ldots, x_{d-m}}_{(r-1)})$ to $S$.

Clearly we can color the new elements of $S'$
so that in total there are exactly $r$ points of the first color, and exactly $r-1$ points
of every other color.

We see that $S'$, $M'$ satisfy the assumptions of Theorem~\ref{thm:colored-tverberg-special}.
It follows that there are $r$ rainbow subsequences $S'_1,\ldots, S'_r$ of $S'$
satisfying $\cl'\emptyset\subsetneq \cl' S'_1\subseteq \cl' S'_2\subseteq \ldots \subseteq S'_r$.

Since the points $x_i$ are coloops and since each one of them was added exactly $(r-1)$-times,
it follows that they cannot contribute to $\bigcap_{i=1}^r \cl' \left(S'_1\right)$.
Consequently, $\bigcap_{i=1}^r \cl \left(S\cap S'_i\right)\neq \cl\emptyset$
and $\cl\emptyset\subsetneq\cl(S\cap S'_1)\subseteq \cl(S\cap S'_2)\subseteq \ldots \subseteq \cl(S\cap S'_r)$.

We conclude that $S_1:=S\cap S_1', S_2:=S\cap S_2', \ldots, S_r:=S\cap S_r'$
are the required subsequences of $S$.

Observe that the reduction is polynomial in $r$, $m$, $u$ and $|S|$.
\end{proof}
\bigskip

Now we can start with the proof of Theorem~\ref{thm:colored-tverberg-special}.
Here we describe the main idea. We let $S_r$ be a rainbow independent subsequence of the maximal rank.
In an ideal case $\cl(S_r)=M$ and we may obtain the remaining 
subsequences $S_1,\ldots, S_{r-1}$ by apply induction on the sequence $S\setminus S_r$ inside $M$.

However, we may be unlucky. It may happen that no such $S_r$
satisfies $\cl(S_r)=M$, see Fig.~\ref{fig:unlucky}.
\begin{figure}[ht]
 \begin{center}
 \includegraphics{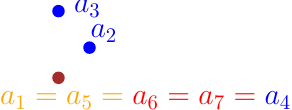}\\
 \medskip
 A situation in which no rainbow $S_r$ satisfies $\cl(S_r)=M$.
 $M=\R^2$, $\cl$ is the affine hull, $r=3$.
 Points $a_2,a_3,a_4$ use the first color (blue),
 $a_1,a_5$ use orange,
 $a_6$, $a_7$ use red.
 \caption{An example without $\cl(S_r)=M$.}\label{fig:unlucky}
 \end{center}
\end{figure}

We see that in this case
we could simply take the subsequence $S'=(a_1,a_4,a_5,a_6,a_7)$
and unify colors blue and red into one color (say violet).
Then $S'$ lives in a submatroid of rank $1$ and satisfies
the conditions of Theorem~\ref{thm:colored-tverberg-special}, so
we may use induction. We obtain subsequences $S_1,S_2,S_3$ of $S'$
satisfying $\cl\emptyset\subsetneq \cl(S_1)\subseteq \cl(S_2)\subseteq \cl(S_3)$.
These are clearly also subsequences of $S$.
Moreover they are not only rainbow in the violet-orange coloring,
but also in the original blue-orange-red coloring. 

In the proof we show that if $\cl(S_r)\neq M$,
we may always resolve the situation by an analogous trick. 

Let us now carry out the technical details. Since we promised an algorithmic solution,
we describe an algorithm that finds the desired subsequences.
\begin{proof}
 First we compute $m'=\rk S$. Since instead of $S$ we can consider
 the subsequence $S'$ formed by the elements
 colored by the first $m'$ colors (while preserving all assumptions of Theorem~\ref{thm:colored-tverberg-special}), 
 we may assume that $M=\cl(S)$.
 
 Now we find an inclusion maximal independent rainbow subsequence $RI_r$ of $S$.
 This can clearly be done in time polynomial in $r$, $u$, $m$ and $|S|$.
 
 We will proceed in the proof by induction on the triple $(r,m,m-\rk RI_r)$ (in lexicographical ordering).
 If $r=1$ or $m=1$ the statement is trivial, so assume $r>1, m>1$.
 
 If $m-\rk RI_r=0$, then $\cl(RI_r)=M$.
 Because $RI_r$ is rainbow, $S\setminus RI_r$ and $M$ satisfy the assumptions of Theorem~\ref{thm:colored-tverberg-special}
 for $r'=r-1$ . By applying induction we obtain $r-1$ disjoint rainbow subsequences $S_1,\ldots, S_{r-1}$ of $S\setminus RI_r$
 with $\cl\emptyset\subsetneq\cl(S_1)\subseteq \cl(S_2)\subseteq\ldots\subseteq \cl(S_{r-1})$.
 If we now set $S_r=RI_r$ we see that $S_1,\ldots, S_r$ are the desired disjoint rainbow subsequences
 with $\cl\emptyset\subsetneq\cl(S_1)\subseteq \cl(S_2)\subseteq\ldots\subseteq \cl(S_{r-1})\subseteq\cl(S_r)$.
 
 Therefore we may assume that 
 \begin{equation}
  \cl(RI_r)\subsetneq M\label{p:step_back}
 \end{equation}

 We would like to increase $RI_r$ by adding a point of a color
 that is not yet used in $RI_r$. Unfortunately, this is not possible without replacing some
 points of $RI_r$ first. Our algorithm uses a cycle to find out which points to replace and how.
 Within the cycle we need to keep track of ``replacement rules'' which makes this part a bit technical.
 Moreover, there are three possibilities what can occur at one iteration of the cycle:
 \begin{enumerate*}[label={\alph*})]
 \item either we construct a larger independent rainbow set $RI_r$, 
 \item we find the desired sets $S_1,\ldots, S_r$ in a smaller submatroid, or
 \item we adjust the replacement rules.
 \end{enumerate*}
 
 \paragraph{The cycle}
 In the $k$th step ($k=0,1,2,\ldots$) of the cycle the replacement rules consist of the following data:
 \begin{enumerate}
  \item set $K_k$ of colors (this set corresponds to colors that we may use while replacing some points),
  \item subsequence $I_k$ of $RI_r$ (eventually we would like to replace the subsequence $I_k$ of $RI_r$ by another sequence $I_k^p$),
  \item for each element $p$ whose color is in $K_k$ and which does not lie in $\cl(I_k)$ a subsequence $I_k^p$ of $S$
        (we want to replace $I_k$ with $I_k^p$, hence increasing the length of our subsequence by one)
 \end{enumerate}
 To simplify the terminology, if $T$ is a subsequence of $S$,
 let $c(T)$ denote the set of all the colors used by elements of $T$.
 If $U$ is a set of colors, let $C_U$ be the subsequence of $S$
 formed by all elements with color from $U$.
 
We want the data to satisfy the following conditions:
\begin{enumerate}[label=({\roman*})]
 \item $c(I_k)\subsetneq K_k$,\label{c:more-colors}
 \item $c(I_k^p)=c(I_k)\cup \{c_k^p\}$ for some  $c_k^p\in K_k\setminus c(RI_r)$,\label{c:rkp-colors}
 \item $|I_k^p|=|I_k|+1$,\label{c:rkp-augmenting-size}
 \item $p\in I_k^p$ and $\cl(I_k^p\setminus \{p\})=\cl(I_k)$\label{c:same_affine_hulls}
 \item $RI_r\cap C_{K_k}=I_k$ and $K_k\not\subseteq c(RI_r)$\label{c:no-further-colors}
\end{enumerate}
Note that conditions~\ref{c:rkp-colors} and \ref{c:rkp-augmenting-size} imply that $I_k^p$ only contains
elements that have the same colors as points in $I_k$ plus one additional point
that has color $c_k^p$, which is not yet present in $RI_r$.

The first step ($k=0$) is easy. We set $I_0:=\emptyset$ and let $K_0$ be all the colors of $S$
except for those already used in $RI_r$. 
No element $p\in C_{K_0}$ is contained in\footnote{$C_{K_0}$ are the elements of $S$ whose
color lies in $K_0$ and we assume that $S$ contains only non-loop elements.} $\cl(I_0)=\cl \emptyset$,
so we need to define the set $I_0^p$ for every such $p$. We simply put $I_0^p:=\{p\}$.

Now we check that the above defined sets satisfy all the prescribed conditions.
Note that by~\eqref{p:step_back}, $S\not\subseteq \cl(RI_r)$. 
This together with the fact that $RI_r$ is independent
implies that $|RI_r|<m$. Since we have $m$ colors, there is a color that is not used in $RI_r$. 
In other words, $K_0$ is nonempty.

Hence conditions~\ref{c:more-colors}--\ref{c:no-further-colors} are satisfied trivially 
(with $c_k^p=c(p)$ in condition~\ref{c:rkp-colors}).

So suppose that the sets $K_k$, $I_k$ and $I_k^p$ are already constructed. 
Since $I_k\subseteq RI_r$ there are three cases that may occur:
\begin{enumerate}[label={\alph*})]
 \item $C_{K_k}\subseteq \cl(I_k)$,\label{c:lower-dimensional}
 \item $C_{K_k}\not\subseteq \cl(RI_r)$ or\label{c:stop-backward-induction}
 \item $C_{K_k}\subseteq \cl(RI_r)$ and $C_{K_k}\not\subseteq \cl(I_k)$.\label{c:continue-backward-induction}
\end{enumerate}

We deal with the particular cases separately:
\subsubsection*{Case~\ref{c:lower-dimensional}: $C_{K_k}\subseteq \cl(I_k)$}
In this case, we may apply the trick we used for Fig.~\ref{fig:unlucky}.
Let us describe it formally.

We set $M':=\cl I_k$ and $m':=\rk(I_k)$. $M$ has rank $m$ and by \eqref{p:step_back} 
we know that $M\not\subseteq \cl(RI_r)$.
It follows that $\rk(RI_r)<m$ and since $I_k\subseteq RI_r$, we also have $m'=\rk(I_k)<m$. 

Condition~\ref{c:more-colors} implies $c(I_k)\subsetneq K_k$, so there is a point $p\in C_{K_k}\setminus C_{c(I_k)}$.

Because $I_k$ is rainbow and independent and $\rk I_k=m'$, $c(I_k)$ has $m'$ distinct elements,
say $k_1,\ldots, k_{m'}$. We define $S':=C_{\{k_1,\ldots, k_{m'}\}}\cup \{p\}$.
In $S'$ we recolor $p$ and all points of color $k_1$ by a new color $z$.

Because $S'\subseteq C_{K_k}$ (we evaluate $C_{K_k}$ with respect to the original coloring), 
the assumption $C_{K_k}\subseteq \cl(I_k)$ (Case~\ref{c:lower-dimensional}) implies
that $S'$ is a sequence of elements from $M'$. Also in $S'$ there are $m'$ colors, at least $r$ elements of color $z$
and at least $r-1$ elements of all the remaining colors. 
Therefore, the assumptions of Theorem~\ref{thm:colored-tverberg-special} are satisfied
for $m'<m$. By induction we obtain the desired disjoint rainbow subsequences $S_1,\ldots, S_r$ of $S'$ (which itself is a subsequence of $S$) 
with $\cl\emptyset\subsetneq \cl(S_1)\subseteq \cl(S_2)\subseteq \ldots \subseteq \cl(S_r)$.
These subsequences are rainbow with respect to the new coloring of $S'$.
By the construction of the new coloring these subsequences
are also rainbow in the original coloring of $S$.

\subsubsection*{Case~\ref{c:stop-backward-induction}: $C_{K_k}\not\subseteq \cl(RI_r)$}
In this case, we construct a new independent rainbow subsequence $RI_r'$ with $|RI_r'|=|RI_r|+1$:
We pick a point $p\in C_{K_k}$ with $p\notin \cl(RI_r)$ 
and set $RI_r':=\left(RI_r\setminus I_k\right)\cup I_k^p$.

Before we show that such $RI_r'$ is a rainbow independent subsequence 
of size $|RI_r|+1$, we prove 
the following auxiliary equality:
\begin{equation}\label{eq:increasing-hull}
 \cl(RI_r') = \cl \left(RI_r\cup\{p\}\right).
\end{equation}

Indeed,
\begin{eqnarray*}
  \cl(RI_r') &=& \cl\left(\left(RI_r\setminus I_k\right)\cup I_k^p\right) \\
				&=& \cl\left(\left(RI_r\setminus I_k\right)\cup (I_k^p\setminus\{p\})\cup\{p\}\right),
\end{eqnarray*}
where the last equality uses the fact that $p\in I_k^p$ from condition~\ref{c:same_affine_hulls}.
Because any closure operator $\cl$ satisfies 
\begin{equation}\label{eq:properties_of_aff}
  \cl\bigl(B\cup C\bigr) = \cl\bigl(B\cup \cl C\bigr)\qquad\text{for any two sets $B,C\subseteq M$},   
\end{equation}
we may rewrite the expression further to
\begin{equation*}
 \cl(RI_r') = \cl\bigl(\left(RI_r\setminus I_k\right)\cup \cl\bigl(I_k^p\setminus\{p\}\bigr)\cup \{p\}\bigr).
\end{equation*}
By condition~\ref{c:same_affine_hulls} $\cl\bigl(I_k^p\setminus\{p\}\bigr)=\cl(I_k)$,
which reduces the equality to:
\begin{equation*}
 \cl(RI_r') =  \cl\Bigl(\bigl(RI_r\setminus I_k\bigr)\cup \cl(I_k)\cup \{p\}\Bigr).
\end{equation*}
Using \eqref{eq:properties_of_aff} again, we obtain
\begin{eqnarray*}
 \cl(RI_r') &=& \cl\Bigl((RI_r\setminus I_k)\cup I_k\cup\{p\}\Bigr)
\end{eqnarray*}
Since $I_k\subseteq RI_r$, Equation~\eqref{eq:increasing-hull} follows.

Using the fact that $I_k\subseteq RI_r$, 
we are now ready to verify that $RI_r'$ is a rainbow independent subsequence with $|RI_r'|=|RI_r|+1$.

\begin{itemize}
\item \textbf{$|RI_r'|=|RI_r|+1$:}
$|RI_r'| = |\left(RI_r\setminus I_k\right)\cup I_k^p|$.
Because $RI_r$ is rainbow, condition\footnote{$c(I_k^p)= c(I_k) \cup \{c_k^p\}$, for some $c_k^p\in K_k\setminus c(RI_r)\subseteq K_k\setminus c(I_k)$}~\ref{c:rkp-colors} implies
that the sequences $RI_r\setminus I_k$ and $I_k^p$ do not share any color.
In particular, they are disjoint and $|RI_r'|=|RI_r\setminus I_k| + |I_k^p|$.
Since $|I_k^p|=|I_k|+1$ (condition~\ref{c:rkp-augmenting-size}), $|RI_r'| = |RI_r\setminus I_k| + |I_k| + 1$.
Because $I_k\subseteq RI_r$, we have $|RI_r'|= |RI_r| + 1$.  
\item \textbf{$RI_r'$ is rainbow:}
$I_k^p$ contains one element of color that is not used in $RI_r$, otherwise it uses the same colors as $I_k$.
Because $RI_r'=\left(RI_r\setminus I_k\right)\cup I_k^p$, we see that $P'_r$ uses exactly $|RI_r|+1$ colors.
This, together with the previous item, yields that $P'_r$ is rainbow.
\item \textbf{$RI_r'$ is independent}:
From the equality~\eqref{eq:increasing-hull} we get $\cl(RI_r') = \cl\left(RI_r \cup\{p\}\right).$
Moreover, we have chosen a point $p$ which satisfies $p\notin \cl(RI_r)$, so $\rk(RI_r') = \rk RI_r + 1.$
Since $RI_r$ was independent and $RI_r'$ has exactly one element more,
the independence of $RI_r'$ follows.
\end{itemize}
Let $RI_r''$ be an inclusion maximal independent rainbow subsequence of $S$
that contains $RI_r'$.
We may now start our algorithm again but this time we replace the maximal independent rainbow subset
$RI_r$ by $RI_r''$.
We have decreased the quantity $(m-\rk C_r)$ and preserved $m$ and $r$.
By induction we obtain the desired disjoint rainbow subsequences $S_1,\ldots, S_r$ with $\cl\emptyset\subsetneq \cl S_1 \subseteq \cl S_2\subseteq \ldots \subseteq  \cl S_r$.

\subsubsection*{Case~\ref{c:continue-backward-induction}: $C_{K_k}\subseteq \cl(RI_r)$ and $C_{K_k}\not\subseteq \cl(I_k)$}
In this case, we show how to construct sets $K_{k+1}$, $I_{k+1}$ and 
for every $p\in C_{K_{k+1}}$ with $p\notin\cl(I_{k+1})$ we construct a subsequence $I_{k+1}^p$.

We choose $I_{k+1}$ to be any inclusion minimal subsequence $I_{k+1}\subseteq RI_r$ satisfying 
\begin{equation}\label{eq:choosing_rk}
C_{K_k}\subseteq \cl I_{k+1}. 
\end{equation}
Because we assume that $C_{K_k}\subseteq \cl(RI_r)$, such set $I_{k+1}$ does exist.
We further define 
\begin{equation}\label{eq:kk}
K_{k+1}:=K_k\cup c(I_{k+1}).  
\end{equation}

Before we construct $I_{k+1}^p$, we prove the following auxiliary claim:

\begin{claim}\label{cl:two}
\begin{equation}
 I_k\subsetneq I_{k+1}\qquad\text{and}\qquad \cl I_k\subsetneq \cl I_{k+1}.
\end{equation}
\end{claim}
\begin{proof}
By condition~\ref{c:more-colors}, $I_k\subsetneq C_{K_k}$.
By Eq.~\eqref{eq:choosing_rk}, we have $\cl I_k \subseteq \cl I_{k+1}$.
By construction both $I_k$ and $I_{k+1}$ are subsequences of the independent sequence 
$RI_r$ which together with the preceding yields $I_k\subseteq I_{k+1}$. 
Condition~\ref{c:more-colors} and the fact that we are in case~\ref{c:continue-backward-induction} yields
$I_k\subseteq C_{K_k}\not\subseteq \cl I_k$.
Since also $C_{K_k}\subseteq \cl I_{k+1}$, we see that $\cl I_{k+1}\neq \cl I_k$ and $I_{k+1}\neq I_k$.
\end{proof}

Now we construct sets $I_{k+1}^p$ for all points $p\in C_{K_{k+1}}$ satisfying $p\notin\cl I_{k+1}$.
Let $p$ be such a point.
By definition of $I_{k+1}$, $C_{K_k}\subseteq \cl I_{k+1}$, so $p$ cannot lie in $C_{K_k}$.
Equation \eqref{eq:kk} implies $c(p)\in \left(K_{k+1}\setminus K_k\right) \subseteq c(I_{k+1})$. 
Because $I_{k+1}\subseteq RI_r$ is a rainbow set\footnote{$RI_r$ is rainbow!}, there exists a unique element $r\in I_{k+1}$ with $c(r)=c(p)$.
Since we assume $p\notin C_{K_k}$, we have $c(r)=c(p)\notin K_k\supseteq c(I_k)$, where the last inclusion follows from condition~\ref{c:more-colors}.
In particular, $c(r)\notin c(I_k)$, hence
\begin{equation}\label{eq:outside}
r\in I_{k+1}\setminus I_k. 
\end{equation}
Since $I_{k+1}$ is an inclusion minimal subsequence of $RI_r$ for which $C_{K_k}\subseteq \cl I_{k+1}$,
there exists an element $q\in C_{K_k}$ such that 
$q\notin\cl\bigl(I_{k+1}\setminus \{r\}\bigr)$. 
Since $q\in C_{K_k}\subseteq \cl I_{k+1}$,
the exchange principle implies $r\in\cl\bigl((I_{k+1}\setminus \{r\})\cup \{q\}\bigr)$.

It easily follows that
\begin{equation}\label{eq:choosing-pqr}
 \cl I_{k+1} = \cl \bigl((I_{k+1}\setminus \{r\})\cup\{q\}\bigr).  
\end{equation}

Claim~\ref{cl:two} together with \eqref{eq:outside} imply that $I_k\subseteq I_{k+1}\setminus \{r\}$. Since $q$
was chosen to satisfy $q\notin \cl\bigl(I_{k+1}\setminus\{r\}\bigr)$, we have $q\notin\cl I_{k}$ as well.
Together with $q\in C_{K_k}$, this implies that $I_k^q$ is defined.
We set\footnote{We note that $I_{k+1}^p$ does depend on the choice of $q$, i.e.,
if we choose another $q\in C_{K_k}$ that satisfies $q\notin\cl\bigl(I_{k+1}\setminus\{r\}\bigr)$,
we obtain a different set $I_{k+1}^p$.}
\begin{equation}\label{eq:defining-rkp}
 I_{k+1}^p:=I_{k+1}\setminus\bigl(I_k\cup \{r\}\bigr)\cup I_k^q\cup\{p\}.
\end{equation}

It remains to show that our assignment satisfies conditions~\ref{c:more-colors}--\ref{c:no-further-colors}.

\begin{itemize}
\item \textbf{Condition~\ref{c:more-colors}:}
By \eqref{eq:kk}, we have $c(I_{k+1})\subseteq K_{k+1}$.
Condition~\ref{c:no-further-colors} implies that $K_k$ contains a color that is not used in $RI_r$
and since $I_{k+1}\subseteq RI_r$, which together with \eqref{eq:kk} yields
$K_{k+1}\neq c(I_{k+1})$. Condition~\ref{c:more-colors} follows.
\item \textbf{Condition~\ref{c:rkp-colors}:}
Condition~\ref{c:rkp-colors} states that $c(I_k^q)=c(I_k) \cup \{c_k^q\}$ for some $c_k^q\in K_k\setminus RI_r$, in particular
$c(I_k)\subseteq c(I_k^q)$. 
Together with the fact that elements $p$ and $r$ have the same color ($c(p)=c(r)$), \eqref{eq:defining-rkp} yields 
$c(I_{k+1}^p) = c\left(I_{k+1}\setminus I_{k}\right)\cup c(I_{k}^q)$. If we now apply condition~\ref{c:rkp-colors} for $I_k^q$
and Claim~\ref{cl:two}, we see that $c(I_{k+1}^p) = c(I_{k+1}) \cup \{c_{k+1}^p\}$, where $c_{k+1}^p = c_k^q$.
Note that $K_k\subseteq K_{k+1}$, hence $c_{k+1}^p\in K_{k+1}\setminus c(RI_r)$. Condition~\ref{c:rkp-colors} follows.

\item \textbf{Condition~\ref{c:rkp-augmenting-size}:}
By definition $I_{k+1}^p = I_{k+1}\setminus \bigl(I_{k}\cup \{r\}\bigr)\cup I_{k}^q \cup \{p\}$.
Because $I_{k+1}$ is a subset of the rainbow set $RI_r$, $I_{k+1}$ is itself rainbow.
Together with $c(I_{k}^q) = c(I_{k}) \cup \{c_{k}^q\}$, where $c_{k}^q\notin c(RI_r)\supseteq c(I_{k+1})$,
this implies that the sets $I_{k+1}\setminus I_{k}$ and $I_{k}^q$ are disjoint. 
Since $r\in I_{k+1}\setminus I_k$ (Equation~\eqref{eq:outside}), $c(p)=c(r)\in c(I_{k+1})\setminus c(I_k)$ and $c(I_k^q)\cap c(RI_r)=c(I_k)$ (conditions~\ref{c:rkp-colors} and \ref{c:no-further-colors}),
we have $p,r\notin I_{k}^q$ and $p,r\notin I_{k}$.
From $p\notin\cl I_{k+1} $ follows $p\notin I_{k+1}$. Since $r\in I_{k+1}$,
we have $|I_{k+1}^p| = |I_{k+1}\setminus I_{k}| - |\{r\}| + |\{p\}| + |I_{k}^q| = |I_{k+1}\setminus I_{k}| + |I_{k}| + 1$,
where the last equality uses the induction hypothesis for $k$.
Claim~\ref{cl:two} then yields $|I_{k+1}^p|=|I_{k+1}| + 1$ as desired.

\item \textbf{Condition~\ref{c:same_affine_hulls}:}
By definition (\eqref{eq:choosing-pqr}) $p\in I_{k+1}^p$, so we only need to verify that $\cl \bigl(I_{k+1}^p\setminus \{p\})\bigr) = \cl I_{k+1}$.
Let us compute.
Using the fact that $q\in I_{k}^q$ from condition~\ref{c:same_affine_hulls} and \eqref{eq:properties_of_aff}
we may rewrite $\cl\bigl(I_{k+1}^p\setminus\{p\}\bigr)$ as follows:
\begin{eqnarray*}\cl\bigl(I_{k+1}^p\setminus\{p\}\bigr) &=& \cl\Bigl(\bigl(I_{k+1}\setminus (I_{k}\cup \{r\})\bigr)\cup I_{k}^q\Bigr)\\
 &=& \cl\Bigl(\bigl(I_{k+1}\setminus (I_{k}\cup \{r\})\bigr)\cup (I_{k}^q\setminus\{q\}\cup\{q\})\Bigr)\\
 &=& \cl\Bigl(\bigl(I_{k+1}\setminus (I_{k}\cup \{r\})\bigr)\cup \cl\bigl(I_{k}^q\setminus\{q\}\bigr)\cup \{q\}\Bigr).\\
\end{eqnarray*}
Now we use condition~\ref{c:same_affine_hulls} for $k$ ($\cl\bigl(I_{k}^q\setminus\{q\}\bigr)= \cl I_{k}$). We obtain
\begin{eqnarray*}
\cl\bigl(I_{k+1}^p\setminus\{p\}\bigr) 
 &=& \cl\Bigl(\bigl(I_{k+1}\setminus (I_{k}\cup \{r\}\bigr)\cup \cl\bigl(I_{k}\bigr)\cup \{q\}\Bigr)\\
 &=& \cl\Bigl(\bigl(I_{k+1}\setminus \{r\}\bigr)\cup\{q\}\Bigr)\\
 &=& \cl I_{k+1},
\end{eqnarray*}
where the last equality follows from \eqref{eq:choosing-pqr}.

\item \textbf{Condition~\ref{c:no-further-colors}:}
By definition $K_{k+1}=K_k\cup c(I_{k+1})$. This implies $C_{K_{k+1}} = C_{K_k}\cup C_{c(I_{k+1})}$.
Hence $RI_r \cap C_{K_{k+1}} = (RI_r\cap C_{K_k})\cup (RI_r\cap C_{c(I_{k+1})})$.
By the induction assumption $RI_r\cap C_{K_k}=I_k$. Because $RI_r\supseteq I_{k+1}$ is rainbow,
$RI_r\cap C_{c(I_{k+1})}=I_{k+1}$. Claim~\ref{cl:two} then implies $RI_r\cap C_{K_{k+1}}=I_{k+1}$ as desired.
Because $K_k\not\subseteq c(RI_r)$ and $K_k\subseteq K_{k+1}$, we have $K_{k+1}\not\subseteq c(RI_r)$ as well.
\end{itemize}
It follows that we may increase $k$ and continue in the loop.

In each step of the cycle we either terminate and output the desired subsequences,
or we construct a sequence $I_{k+1}$ whose rank is strictly larger than the rank of $I_k$ (Claim~\ref{cl:two}).
Since the rank of $I_{k+1}$ is from above bounded by $\rk(M)$ it follows that
the loop terminates after at most $\rk(M)$ iterations.

Verifying that all other steps can be done in time polynomial in $r$, $m$, $u$ and $|S|$ and 
that they are repeated only polynomial number of times is easy.
\end{proof}

\section{Open problems}
Rota basis conjecture~\cite{rota-1994} is a well known problem in matroid theory which has a close connection to our
colorful matroidal Tverberg's theorem. Let us restate it so that the similarity is clearly visible.
\begin{conjecture}
Let $M$ be a matroid of rank $m$. Let $S$ be a sequence of $m^2$ elements colored
by $m$ colors such that points of each color form a basis.
Do there always exist $m$ pairwise disjoint rainbow subsequences $S_1,\ldots, S_m$ of $S$
with $\cl S_1=\cl S_2=\ldots = \cl S_r=M$?
\end{conjecture}
In its full generality the conjecture has only been verified for $m=1,2,3$~\cite{chan-1995}.
The conjecture is also known to be true in several special cases~\cite{Geelen-2006,Onn-1997,Glynn-2010}.
Proof of Theorem~\ref{thm:colored-tverberg-special} indicates the difficulties that appear if one
tries to prove Rota's basis conjecture purely combinatorially.

\newcommand{\etalchar}[1]{$^{#1}$}


\begin{thebibliography}{GMP{\etalchar{+}}16}

\bibitem[BGR15]{anotherMatroidalTverberg}
I.~Bárány, Kalai G., and Meshulam R.
\newblock {A Tverberg type theorem for matroids}.
\newblock {\em ArXiv e-prints}, 2015.
\newblock Available online at \url{http://arxiv.org/abs/1607.01599}.

\bibitem[BMZ15]{Blagojevic:Optimal_Tverberg}
P.~V.~M. Blagojevi{\'c}, B.~Matschke, and G.~M. Ziegler.
\newblock Optimal bounds for the colored {T}verberg problem.
\newblock {\em J. Eur. Math. Soc.}, 17(4):739--754, 2015.

\bibitem[Cha95]{chan-1995}
W.~Chan.
\newblock An exchange property of matroid.
\newblock {\em Discrete Mathematics}, 146(1):299 -- 302, 1995.

\bibitem[dL02]{deLongueville:exposition}
M.~de~Longueville.
\newblock Erratum to: ``{N}otes on the topological {T}verberg theorem''.
\newblock {\em Discrete Math.}, 247(1-3):271--297, 2002.

\bibitem[{Fri}15]{Frick}
F.~{Frick}.
\newblock {Counterexamples to the topological Tverberg conjecture}.
\newblock {\em ArXiv e-prints}, 2015.
\newblock Available online at \url{http://arxiv.org/abs/1502.00947}.

\bibitem[GH06]{Geelen-2006}
J.~Geelen and P.~J. Humphries.
\newblock Rota's basis conjecture for paving matroids.
\newblock {\em SIAM J. Discrete Math.}, 20(4):1042--1045, 2006.

\bibitem[Gly10]{Glynn-2010}
D.~G. Glynn.
\newblock The conjectures of {Alon--Tarsi} and {Rota} in dimension prime minus
  one.
\newblock 24(2):394--399, 2010.

\bibitem[GMP{\etalchar{+}}15]{ourKuhnel}
X.~Goaoc, I.~Mabillard, P.~Paták, Z.~Patáková, M.~Tancer, and U.~Wagner.
\newblock On generalized {H}eawood inequalities for manifolds: a {V}an
  {K}ampen-{F}lores-type nonembeddability result.
\newblock {\em Extended abstract in Proceedings of SoCG'15}, 2015.

\bibitem[GMP{\etalchar{+}}16]{ourKuhnelArXiv}
X.~Goaoc, I.~Mabillard, P.~Paták, Z.~Patáková, M.~Tancer, and U.~Wagner.
\newblock On generalized {H}eawood inequalities for manifolds: a {V}an
  {K}ampen-{F}lores-type nonembeddability result.
\newblock {\em preprint on arXiv:}, 2016.

\bibitem[HR94]{rota-1994}
R.~Huang and G.-C. Rota.
\newblock On the relations of various conjectures on latin squares and
  straightening coefficients.
\newblock {\em Discrete Mathematics}, 128(1):225 -- 236, 1994.

\bibitem[MW14]{Mabillard}
I.~Mabillard and U.~Wagner.
\newblock Eliminating {T}verberg points, {I.} {A}n {A}nalogue of the {W}hitney
  trick.
\newblock {\em Proceedings of the Thirtieth Annual Symposium on Computational
  Geometry (New York, NY, USA), SOCG'14, ACM}, pages 171--180, 2014.

\bibitem[MW15]{Uli:TverbergThree}
I.~{Mabillard} and U.~{Wagner}.
\newblock {Eliminating Higher-Multiplicity Intersections, {I}{I}{I}.
  {C}odimension~2}.
\newblock {\em preprint on arXiv: 1601.00876}, 2015.

\bibitem[Onn97]{Onn-1997}
S.~Onn.
\newblock A colorful determinantal identity, a conjecture of {Rota}, and
  {Latin} squares.
\newblock 104(2):156--159, February 1997.

\bibitem[{\"O}za87]{Ozaydin:Tverberg}
M.~{\"O}zaydin.
\newblock Equivariant maps for the symmetric group.
\newblock {\em Unpublished manuscript}, 1987.
\newblock Available online at \url{http://digital.library.wisc.edu/}
  \url{1793/63829}.

\bibitem[Pat15]{patak:thesis}
P.~Pat{\'{a}}k.
\newblock {\em Using algebra in geometry}.
\newblock PhD thesis, Charles University, 2015.
\newblock avaible online at \url{http://kam.mff.cuni.cz/~patak/thesis.pdf}.

\bibitem[Sar00]{Sarkaria:Tverberg}
K.~S. Sarkaria.
\newblock Tverberg partitions and {B}orsuk-{U}lam theorems.
\newblock {\em Pacific J. Math.}, 196(1):231--241, 2000.

\bibitem[Tve66]{Tverberg}
H.~Tverberg.
\newblock A generalization of {R}adon's theorem.
\newblock {\em J. London Math. Soc.}, 41:123--128, 1966.

\end{thebibliography}
\end{document}